\documentclass[10pt]{article}
\usepackage{amsmath,amssymb,mathrsfs}
\usepackage{subfigure}
\usepackage{graphicx,color,url}
\usepackage{epsfig}
\usepackage{algorithm}
\usepackage{algorithmic}
\usepackage{amsthm}

\newtheorem{theorem}{Theorem}[section]

\newtheorem{definition}[theorem]{Definition}
\newtheorem{remark}[theorem]{Remark}
\newtheorem{lemma}[theorem]{Lemma}
\newtheorem{example}{Example}[section]

\newcommand{\mt}{\mathcal{T}}
\newcommand{\st}{\mathcal{S}(d_{1}, d_{2}, \alpha, \beta, \mathcal{T})}
\newcommand{\m}{\mathcal{M}}

\begin{document}

\title{On the dimension of splines spaces over T-meshes with smoothing
cofactor-conformality method}
\author{X. Li\\
Department of Mathematics, USTC,\\
Hefei, Anhui Province 230026, P. R. China. \\
lixustc@ustc.edu.cn}

\maketitle

\begin{abstract}
The present paper provides a general formula for the dimension of spline space over T-meshes using smoothing
cofactor-conformality method. And we introduce a new notion, \textbf{Diagonalizable T-mesh}, over which the dimension
formula is only associated with the topological information of the T-mesh. A necessary and sufficient condition for characterization
a diagonalizable T-mesh is also provided. Using this technique, we find that the
dimension is possible instable under the condition of~\cite{LiWaZh06} and
we also provide a new correction theorem.
\end{abstract}

\textbf{Keywords: } Dimension, T-splines, Spline space, T-mesh, smoothing cofactor-conformality method

\section{Introduction}
NURBS (Non-Uniform Rational B-Spline) is the standard for generating and
representing free-form curves and surfaces, which is a basic tool for using
in CAD~\cite{Farin02} and is also a desirable tool for iso-geometric analysis~\cite{Cottrell:2009rp}.
A well-known and significant disadvantage of NURBS is that it is based on a tensor product
structure with a global knot insertion operation. It is desirable to generalize NURBS to
spline space which can hander hanging nodes.

Spline spaces over T-meshes $\st$ were first introduced in~\cite{htspline1}, which is a bi-degree $(d_{1}, d_2)$ piecewise
polynomial spline space over T-mesh $\mt$ with smoothness order
$\alpha$ and $\beta$ in two directions. Spline space over T-meshes have
been applied in fitting~\cite{htspline2}, stitching~\cite{htspline3},
simplification~\cite{htspline4}, isogeometric analysis~\cite{iga_pht1},~\cite{iga_pht2},
solving elliptic equations~\cite{pht_pde} and spline space over triangulations with hanging nodes~\cite{Schumaker12_ca}. Spline over T-mesh has several advantages. For example,
it has a simple local refinement which will
never introduce additional refinement according to the definition.
And it is a polynomial in each face which has simple and efficient
integration for numerical analysis.

NURBS with local refinement is discovered before spline over T-mesh,
called ''T-spline", which are defined on a T-mesh by certain collections
of B-splines functions defined on the mesh~\cite{SeZhBaNa03},~\cite{SeCaFiNoZhLy04}.
T-splines have been proved to be a powerful free-form geometric shape technology that solves most of the
limitations inherent in NURBS. T-splines have several advantages
over NURBS such as local refinement~\cite{SeCaFiNoZhLy04} and
watertightness~\cite{Sederberg08}, and they are forward and
backward compatible with NURBS. These capabilities make T-splines
attractive both for CAD and also desirable for iso-geometric
analysis~\cite{Bazilevs2009}. However, it is very difficult to unravel the
mystery of T-spline spaces which have important implications in establishing
approximation, stability, and error estimates~\cite{BaBeCoHuSa06}. And till now,
only a sub-class of T-spline space, analysis-suitable T-splines space~\cite{LiZhSeHuSc10}, \cite{ScLiSeHu10}, \cite{LiSc10}
has been discovered using the technology from spline space over T-meshes.

Thus, it is very important to understand the spline space of piecewise
polynomials of a given smoothness on a T-mesh, which foundation but non-trivial step is
to calculate the dimension of the space. Till now, many different methods have been applied to tackle
these issues, such as B-net~\cite{htspline1}, minimal determining set~\cite{Schumaker12_nn}, smoothing
cofactor-conformality method~\cite{LiWaZh06} and homological technique~\cite{Mourrain}. B-net and minimal
determining set methods are suitable for spline space with reduced regularity, i.e., the degrees for the
polynomial is larger enough than the smoothness order. Smooth cofactor-conformality method~\cite{wangrh}, \cite{sch_scm}
is a power tool for calculating the dimension of spline space in multi-variate splines theory. It transfers the
smoothness conditions into algebraic forms and calculate the dimension using linear algebraic tools. Homological
technique is another power tool for calculating the dimension using the similar idea except regarding the smoothness
conditions as the kernel of a certain linear maps. The present paper is focusing on smoothing cofactor-conformality method.

All the existing dimension results are forcing the T-mesh to be special nesting structure, such as hierarchical,
regular T-subdivision and no cycles~\cite{Schumaker12_nn}. With such structure, one can calculate the dimension level by level.
Our approach is based on a different point of view. As smoothing cofactor-conformality method can convert the smoothness conditions
into algebraic forms, so we directly focus on the algebraic forms and study the condition, under which the
dimension of the linear system doesn't associated with knot values. And then we use the condition to
find a new notion, diagonalizable T-meshes, is the corresponding T-meshes. We believe that this new notion is the key
condition to compute the dimension using smoothing cofactor-conformality method. The main contribution of the present paper
includes,
\begin{itemize}
  \item We provide a general formula for the dimension of spline space over any regular T-meshes without holes
  using smoothing cofactor-conformality method.
  \item We provide a new notion, diagonalizable T-mesh, over which the dimension formula is only associated
  with the topological information of the T-mesh. We also provide a necessary and sufficient condition for characterization
  diagonalizable T-meshes;
  \item We discover that the dimension is possible instable under the condition of~\cite{LiWaZh06} and we also provide a new
  correction theorem.
\end{itemize}

The remaining paper is structured as follows. Pertinent background
on spline space over T-mesh is reviewed in Section 2. In Section 3,
we review how to use smoothing cofactor-conformality method to analysis
the dimension of spline space over T-meshes. In section 4, we provide a condition under which
we can calculate the dimension of spline space over T-mesh without considering the knot values. In section 5,
we first show that the dimension result in~\cite{LiWaZh06} is not correct and we provide
a new correction based on this new technology. The
last section is conclusion and future work.

\section{T-mesh and spline over T-mesh}

In this section, we briefly review the notion of spline spaces over
T-meshes and the related dimension results of the corresponding
spline space.
\subsection{T-mesh}
\label{sec:tmesh} A T-mesh $\mt$ is a collection of axis-aligned rectangles $F_{i}$ such
that the interior of the domain $\Omega$ is $\cup F_{i}$, and the distinct rectangles
$F_{i}$ and $F_{j}$ can only intersect at points on their edges.
The rectangles $F_{i}$ are also called the \emph{face} or \emph{cell} of the T-mesh. The
vertices of the rectangles are called the \emph{nodes} or \emph{vertices} for a T-mesh.
The line segment connecting two adjacent vertices on a grid line is called an
\emph{edge} of the T-mesh. T-meshes include tensor-product
meshes as a special case. However, in contrast to tensor-product meshes,
T-meshes are allowed to have \emph{T-junctions}, or \emph{T-nodes}, which are vertices of one rectangle
that lies in the interior of an edge of another rectangle.
The domain $\Omega$ need not be rectangular, which may have holes, concave corners. For example
the T-mesh in Figure~\ref{fig:tmeshexample}, the grey region is a hole and vertices $V_{38}$ and $V_{88}$ are
both concave corners. In the present paper, we require
the T-meshes to be \textbf{regular} and \textbf{without holes}. Here regular means that
the set of all rectangles for a T-mesh containing a vertex has a connected interior~\cite{Schumaker12_cagd}.

\begin{figure}[htbp]
\centering
~\\[-1ex]
\subfigure {\includegraphics [width=2.3in]{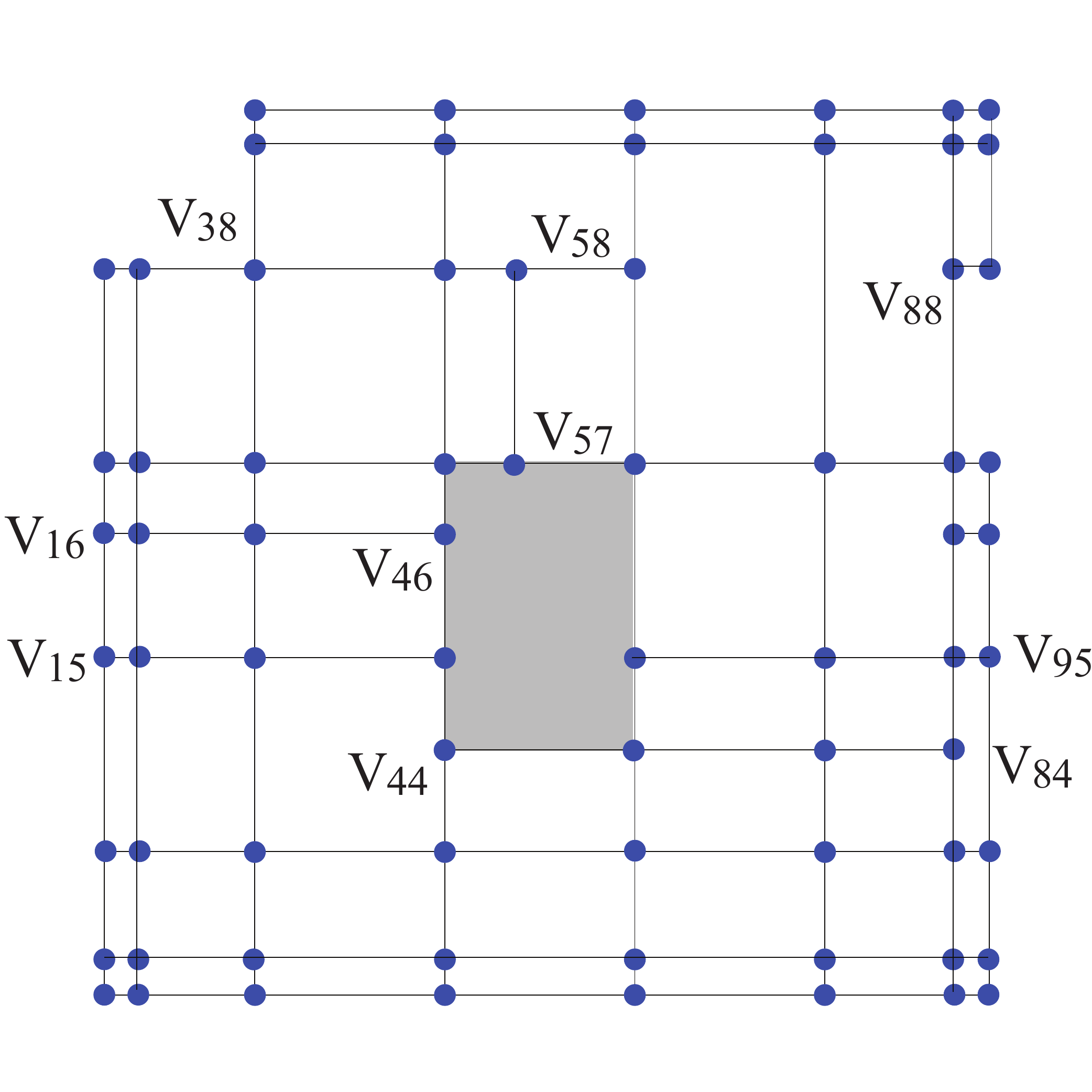}}
~\\[-1ex]
\caption{A T-mesh.} \label{fig:tmeshexample}
~\\[-1ex]
\end{figure}

The vertices, edges can be divided into two parts. If a vertex is on the boundary grid
line of the T-mesh, then is called a \emph{boundary vertex}.
Otherwise, it is called an \emph{interior vertex}. If both vertices of an edge are boundary
vertices, then it is called a \emph{boundary edge}; otherwise it is
called an \emph{interior edge}. An \emph{l-edge} is a line
segment which consists of several interior edges. It is the longest
possible line segment, which interior edges are connected and two
end points being T-junctions or boundary vertices. l-edges have three different classes.
If the two end vertices of a l-edge are interior vertices, then the l-edge is called \emph{interior l-edge}.
If two end vertices of a l-edge are both boundary vertices, then the l-edge is called a \emph{cross-cut}.
Otherwise, if one end vertex is boundary vertex and the other is interior vertex, then the l-edge is called
a \emph{ray}. A \emph{mono-vertex} is the intersection vertex of an interior l-edge and a cross-cut or a ray
and a \emph{free-vertex} is the intersection between cross-cuts and rays.
For example, in Figure~\ref{fig:tmeshexample}, vertices $V_{53}$,
$V_{44}$ and $V_{46}$ are interior vertices, and $V_{15}$, $V_{16}$
and $V_{95}$ are boundary vertices. The l-edge $V_{15}V_{95}$
is a cross-cut, while $V_{16}V_{46}$ is a ray, and $V_{44}V_{84}$
and $V_{57}V_{58}$ are interior l-edges.

For later use we introduce some notations for a T-mesh as shown in Table~\ref{tab:1}

\begin{table}
  \centering
  \caption{Notations for a T-mesh}\label{tab:1}
  \begin{tabular}{l l}
  \hline
  $F$ & number of faces in $\mt$\\
  $E^{h}$ & number of horizonal interior edges in $\mt$\\
  $E^{v}$ & number of vertical interior edges in $\mt$\\
  $V$ & number of interior vertices in $\mt$\\
  $C^{h}$ & number of horizonal cross-cuts in $\mt$ \\
  $C^{v}$ & number of vertical cross-cuts in $\mt$ \\
  $T^{h}$ & number of horizonal interior l-edges in $\mt$\\
  $T^{v}$ & number of vertical interior l-edges in $\mt$\\
  $n_{e}$ & number of interior l-edges in $\mt$ ($T^{h} + T^{v}$)\\
  $V^{+}$ & number of free-vertices in $\mt$ \\
  $N^{h}$ & minimal integer larger or equal to $\frac{d_{1} + 1}{ d_{1} - \alpha}$\\
  $N^{v}$ & minimal integer larger or equal to $\frac{d_{2} + 1}{ d_{2} - \beta}$ \\
  $n_{c}$ & $(d_{1} - \alpha)(d_{2} - \beta)(V - V^{+})$\\
  $n_{r}$ & $(d_{1}+1)(d_{2}-\beta)T^{h} + (d_{2}+1)(d_{1} - \alpha)T^{v}$ \\
  \hline
  \end{tabular}
\end{table}

\subsection{Spline space over T-mesh}
Given a T-mesh $\mathcal{T}\in \mathbb{R}^2$, let $\mathcal{F}$
denote all the cells in $\mathcal{T}$ and $\Omega$ the region
occupied by all the cells in $\mathcal{T}$. The
bi-degree $(d_{1}, d_2)$ polynomial spline space over T-mesh
$\mathcal{T}$ with smoothness order $\alpha$ and $\beta$ is defined as
\begin{equation} \mathcal{S}(d_{1}, d_{2}, \alpha, \beta, \mathcal{T}) :=
\Big\{ f(x, y) \in C^{\alpha, \beta}(\Omega)\Big| f|_\phi \in
\mathrm{P}_{d_{1}d_{2}}, \forall\phi \in \mathcal{F}\Big\},\nonumber
\end{equation}
where $\mathrm{P}_{d_{1}d_{2}}$ is the space of all the polynomials
with bi-degree $(d_{1}, d_{2})$, and $C^{\alpha, \beta}(\Omega)$ is
the space consisting of all the bivariate functions which are
continuous in $\Omega$ with order $\alpha$ along $x$ direction and
with order $\beta$ along $y$ direction. It is obvious that
$\mathcal{S}(d_{1}, d_{2}, \alpha, \beta, \mathcal{T})$ is a linear
space, which is called the spline space over the given T-mesh
$\mathcal{T}$.

Until now, several articles have been studied to analysis the dimension of
the spline space over some special families of T-meshes.
\begin{itemize}
  \item \textbf{Reduced regularity:} \\
  In 2006, \cite{htspline1} studied the dimension of the spline space
  under the constrains that the order of the smoothness is less than half of the degree of the
  spline functions. According to Theorem 4.2 in \cite{htspline1}, it follows that if
  $d_{1} \geq 2\alpha + 1$ and $d_{2} \geq 2\beta + 1$,
  \begin{align*}
  \label{dim3311} \dim\mathcal{S}(d_{1}, d_{2}, \alpha, \beta,
  \mathcal{T}) = & F(d_{1}+1)(d_{2}+1) - E_{h}(d_{1}+1)(\beta+1) -
  \nonumber \\
  & E_{v}(d_{2} + 1)(\alpha + 1) + V(\alpha+1)(\beta+1),
  \end{align*}
  where $F$, $E_{h}$, $E_{v}$, and $V$ are defined in Table~\ref{tab:1}.
  \cite{Schumaker12_cagd} also proved this result using minimal determining set method.
  And later, \cite{buffa} analysis a special T-spline with reducing regularity using the
  dimension in~\cite{htspline1}.
  \item \textbf{Enough mono-vertices} \\
  In 2006, \cite{LiWaZh06} calculated the dimension of spline space
  over a T-mesh if each interior l-edges have enough mono-vertices. In the T-mesh, if the interior
  of each horizontal interior l-edge has at least $N^{h} - 2$
  mono-vertices and the interior of each vertical interior edge segment has at least
  $N^{v} - 2$ mono-vertices, then the dimension of spline space over the T-mesh is,
  \begin{align*}
  \dim\mathcal{S}(d_{1}, d_{2}, \alpha, \beta, \mathcal{T}) = & (d_{1}+ 1)(d_{2} + 1) +
  (C_{h} - T_{h})(d_{1} + 1)(d_{2} - \beta) + \nonumber \\
  & (C_{v} - T_{v})(d_{2} + 1)(d_{1} - \alpha) + V(d_{1} - \alpha)(d_{2} - \beta),
  \end{align*}
  here $C_{h}$, $C_{v}$, $T_{h}$, $T_{v}$ and $V$ are defined in Table~\ref{tab:1}. However, we will show
  that the condition in this paper is not right.
  \item \textbf{Instability: } \\
  In 2011, \cite{xinli,instablity2} discovered that the dimension of the
  associated spline space is instability over some particular
  T-meshes, i.e, the dimension is not only associated with the
  topological information of the T-mesh but also associated with the
  geometric information of the T-mesh.
  \item \textbf{Analysis-suitable T-splines:}\\
  In 2011, \cite{LiZhSeHuSc10,ScLiSeHu10} provided
  a mildly restricted subset of T-splines, which optimized to meet
  the needs of both design and analysis. And~\cite{LiSc10} compute the dimension of
  the spline space $\mathcal{S}(d,d,d-1,d-1,\mathcal{T})$ if the T-mesh $\mt$ is an extended
  T-mesh of an analysis-suitable T-mesh.
  \item \textbf{Regular T-subdivision:} \\
  \cite{Mourrain} studied the dimension for spline space $\st$
  when the T-mesh is a regular T-subdivision by exploiting homological techniques, which is a special case of the present paper.
  \item \textbf{Special hierarchical T-mesh:} \\
  \cite{mengwu} provided the dimension for spline space $\mathcal{S}(d,d,d-1,d-1,\mathcal{T})$
  over a special hierarchical T-mesh using homological algebra technique.
\end{itemize}

\section{Smoothing cofactor-conformality method}
In this section, we will
review one of the main methods, smooth cofactor-conformality
method introduced in~\cite{wangrh} and~\cite{sch_scm}, for computing the dimension of spline space
over T-meshes.

In the theory of multi-variate splines, in order to calculate the
dimension of a spline space, one first needs to transfer the
smoothness conditions into algebraic forms.

Referring to Figure~\ref{fig:smoothv}, for any interior vertex $v_{i, j} =
(x_i, y_j)$, suppose the four surrounding bi-degree $(d_{1}, d_{2})$ polynomial
patches are $p_{i, j}^{k}(x, y), k = 0, 1, \dots, 3$ respectively (if the vertex $v_{i, j}$ is a T-junction, then some
of the polynomial patches are identical). For example for the left
T-junction in Figure~\ref{fig:smoothv} b, patches $p_{i, j}^{1}(x,
y)$, $p_{i, j}^{2}(x, y)$ are identical.
\begin{figure}[htbp]
\begin{center}
\includegraphics[width=0.9\textwidth]{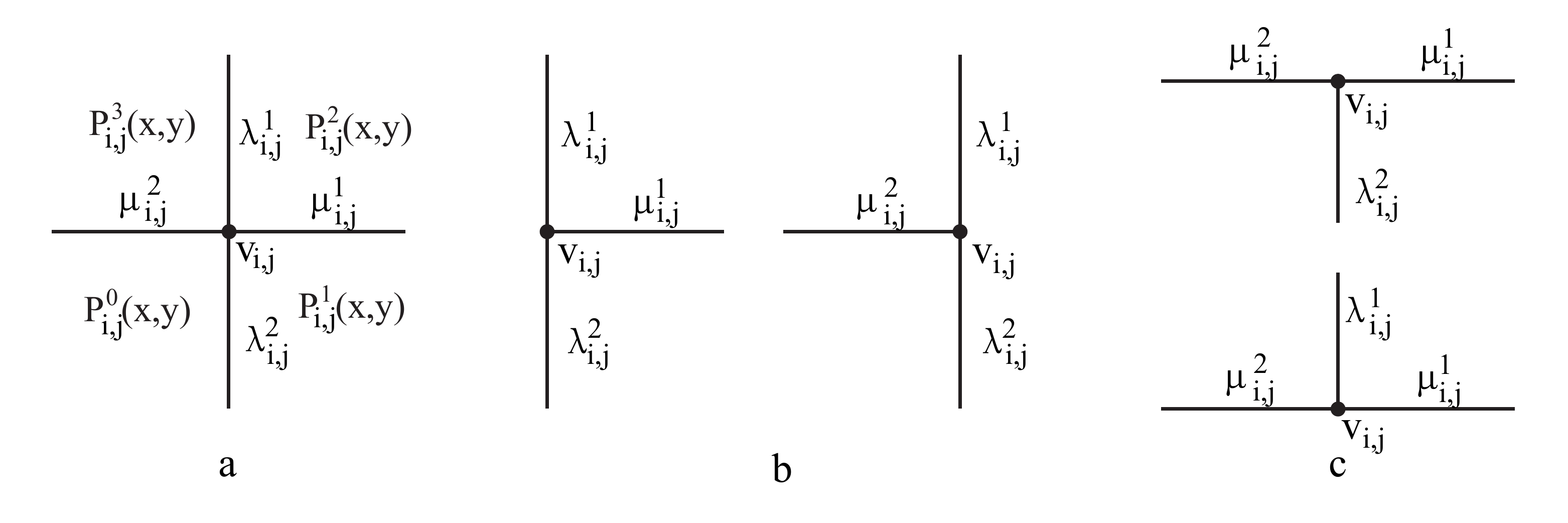}
\end{center}
\caption{smoothing cofactors around a vertex.\label{fig:smoothv}}
\end{figure}

As $p_{i, j}^{0}(x, y)$ and $p_{i, j}^{1}(x, y)$ are $C^{\alpha}$
continuity, so there exists a bi-degree $(d_{1} - \alpha - 1, d_{2})$ polynomial $\lambda_{i,
j}^{2}(y)$ such that
\begin{equation}
\label{equ:0}
p_{i, j}^{1}(x, y) - p_{i, j}^{0}(x, y) = \lambda_{i, j}^{2}(x,y)(x -
x_{i})^{\alpha+1},
\end{equation}
Here $\lambda_{i, j}^{2}(x, y)$ is called \textbf{edge cofactor} for the
common edge of patches $p_{i, j}^{0}(x, y)$ and $p_{i, j}^{1}(x,
y)$. If two patches are identical, then the edge cofactor is zero.

Similarly, there also exist bi-degree $(d_{1} - \alpha - 1, d_{2})$ polynomial $\lambda_{i,
j}^{1}(x, y)$, bi-degree $(d_{1}, d_{2}-\beta-1)$ polynomials $\mu_{i, j}^{1}(x, y)$ and $\mu_{i, j}^{2}(x, y)$, such that
\begin{align}
\label{equ:1}
p_{i, j}^{2}(x, y) - p_{i, j}^{1}(x, y) &= \mu_{i, j}^{1}(x, y)(y - y_{j})^{\beta+1}, \\
\label{equ:2}p_{i, j}^{3}(x, y) - p_{i, j}^{2}(x, y) &= -\lambda_{i, j}^{1}(x, y)(x - x_{i})^{\alpha+1}, \\
\label{equ:3}p_{i, j}^{0}(x, y) - p_{i, j}^{3}(x, y) &= -\mu_{i, j}^{2}(x, y)(y -
y_{j})^{\beta+1}.
\end{align}
Sum with all these equations, we have
\begin{equation}
(\lambda_{i, j}^{1}(x, y) - \lambda_{i, j}^{2}(x, y))(x - x_{i})^{\alpha+1} =
(\mu_{i, j}^{1}(x, y) - \mu_{i, j}^{2}(x, y))(y - y_{j})^{\beta+1}.
\end{equation}

Since $(x - x_{i})^{\alpha+1}$ and $(y - y_{j})^{\beta+1}$ are prime to each
other, so there exist bi-degree $(d_{1}-\alpha-1, d_{2}-\beta-1)$ polynomial $d_{i, j}(x, y)$, such that,
\begin{equation}
\label{equ:edge} \lambda_{i, j}^{1}(x, y) - \lambda_{i, j}^{2}(x, y) =
d_{i, j}(x, y)(y - y_{j})^{\beta+1}, \quad \mu_{i, j}^{1}(x, y) - \mu_{i, j}^{2}(x, y)
= d_{i, j}(x, y)(x - x_{i})^{\alpha+1}.\nonumber
\end{equation}

Let $$d_{i, j}(x, y) = \sum_{p = 0}^{d_{1} - \alpha - 1}\sum_{q = 0}^{d_{2} - \beta - 1}d_{i, j}^{p, q}(x - x_{i})^{p}(y - y_{j})^{q},$$
We call these $d_{i, j}^{p, q}$ \textbf{vertex cofactor}. Denote $\hat{\textbf{d}}_{i, j}$ to be a vector contains all coefficients $d_{i, j}^{p, q}$, i.e.,
$$\hat{\textbf{d}}_{i, j} = (d_{i, j}^{0, 0}, d_{i, j}^{0, 1}, \dots, d_{i, j}^{d_{1} - \alpha - 1, d_{2} - \beta - 2}, d_{i, j}^{d_{1} - \alpha - 1, d_{2} - \beta - 1}).$$

For boundary vertices, there is a little different. Since in this case,
we only have parts of the four equations such as~\eqref{equ:0},\eqref{equ:1},\eqref{equ:2},\eqref{equ:3}. So we don't need
to assign the bi-degree $(d_{1}-\alpha-1, d_{2}-\beta-1)$ polynomial $d_{i, j}(x, y)$ for the boundary vertex. Instead, we
will assign the edge cofactor for the corresponding edge. For example, in Figure~\ref{fig:smoothv1}a, we need two edge cofactors
$\lambda_{i, j}^{1}$ and $\mu_{i, j}^{1}$ and for the figure b, we need only one edge cofactors
$\mu_{i, j}^{1}$.
\begin{figure}[htbp]
\begin{center}
\includegraphics[width=5.0in]{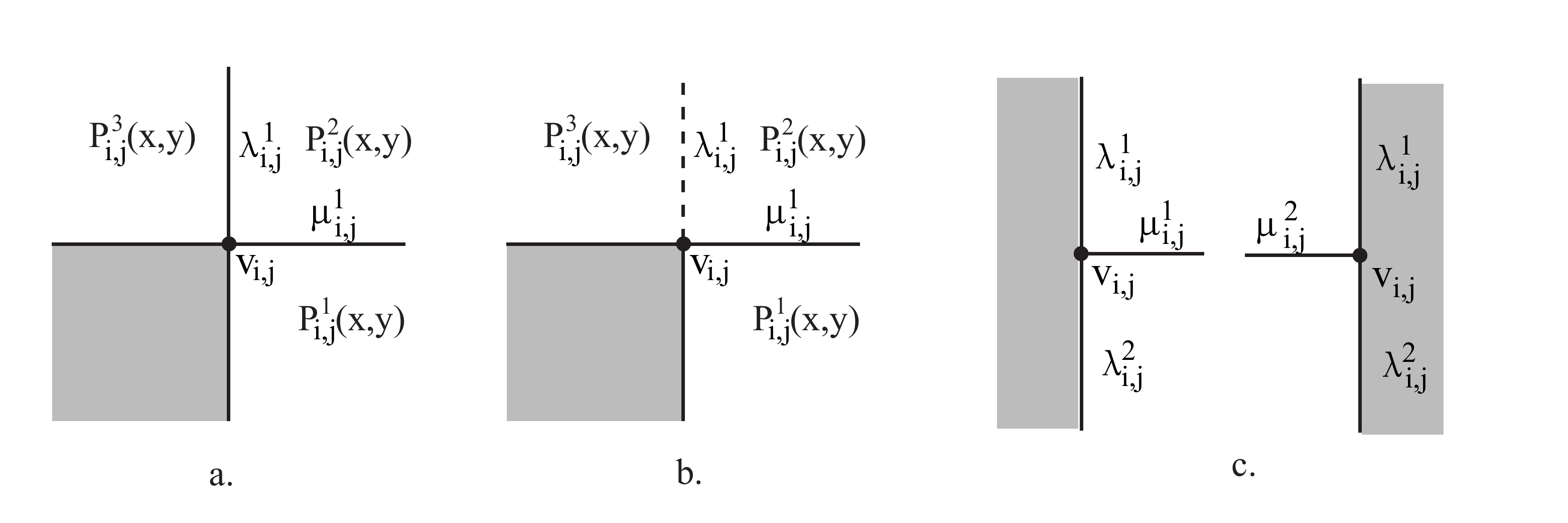}
\end{center}
\caption{smoothing cofactors around a boundary vertex. The grey regions are outside of the T-mesh. \label{fig:smoothv1}}
\end{figure}

The vertex cofactors for the interior vertices are not totally free, there are other
constrains for the continuity condition along each l-edge in
the T-mesh.

\begin{figure}[htbp]
\centering
~\\[-1ex]
\subfigure [a.]{\includegraphics [width=4.0in]{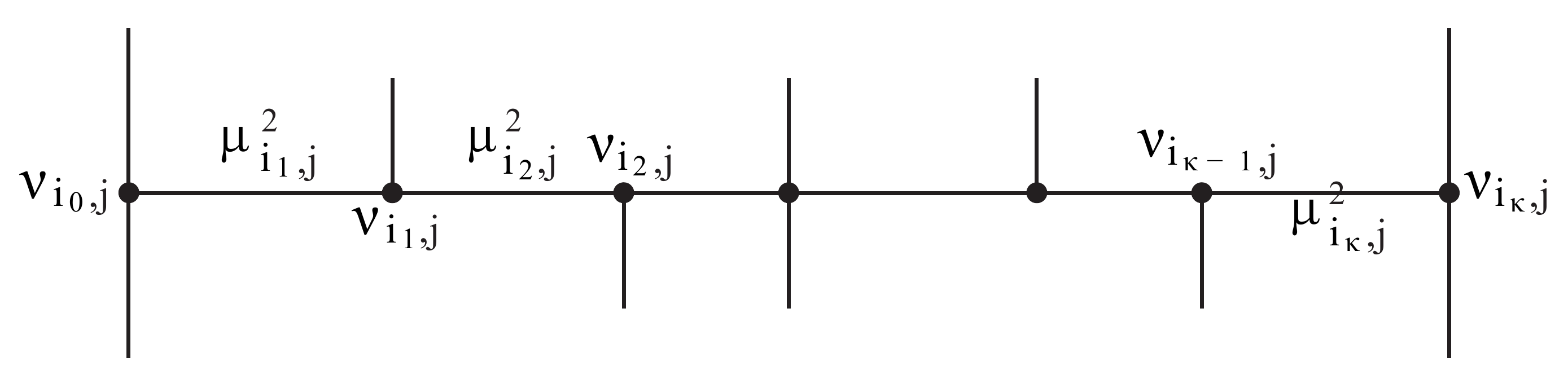}}
~\\[-1ex]
\caption{smoothing cofactors along a horizonal edge
segment.\label{fig:smoothe}}
\end{figure}

We first consider a horizontal \emph{interior l-edge} referring to Figure~\ref{fig:smoothe}
with $k + 1$ vertices $\nu_{i_{t}, j}, t = 0, 1, \dots, k$.
According to equation~\eqref{equ:edge}, we have
\begin{align}
\mu_{i_{0}, j}^{1}(x, y) - 0 &= d_{i_{0}, j}(x, y)( x - x_{i_{0}})^{\alpha+1}, \nonumber\\
\dots \nonumber\\
\label{equ:ledge}\mu_{i_{t}, j}^{1}(x, y) - \mu_{i_{t}, j}^{2}(x, y) &= d_{i_{t}, j}(x, y)( x - x_{i_{t}})^{\alpha+1}, \\
\dots \nonumber\\
0 - \mu_{i_{k}, j}^{2}(x, y) &= d_{i_{k}, j}(x, y)( x - x_{i_{k}})^{\alpha+1}.\nonumber
\end{align}
and $$\mu_{i_{t-1}, j}^{2}(x, y) = \mu_{i_{t}, j}^{1}(x, y).$$ Sum all these
equations, we have the following equation,
\begin{equation}
\label{equ:con0} \sum_{t=0}^{k} d_{i_{t}, j}(x, y)( x - x_{i_{t}})^{\alpha+1} =
0.
\end{equation}
Similarly, the constrains for a vertical interior l-edge is
\begin{equation}
\label{equ:con1} \sum_{t=0}^{l} d_{i, j_{t}}(x, y)( y - y_{j_{t}})^{\beta+1} =
0.
\end{equation}
The above equations are called \textbf{edge conformality
conditions}.

\begin{lemma}\label{lemma:dim}Assume each $x_{i_{t}}$ and $y_{j_{t}}$
are distinct, then the dimensions of solution space of equation~\eqref{equ:con0} and~\eqref{equ:con1}
are $(d_{2}-\beta)(l(d_{1}-\alpha)-d_{1}-1)_{+}$ and $(d_{1}-\alpha)(k(d_{2}-\beta)-d_{2}-1)_{+}$ respectively.
\end{lemma}
\begin{proof}
See~\cite{LiWaZh06} for more details.
\end{proof}

\begin{remark}If the number of the vertices in a horizontal l-edge is less than
$N^{h}$, or the number of the vertices in a vertical l-edge is less than
$N^{v}$, then the l-edge will not contribute the dimension of the spline space, i.e.,
we can delete the l-edge without altering the spline space. Thus, we called such l-edge \textbf{vanished l-edge}.
In the following, we assume that all l-edges are not vanished l-edges.
\end{remark}

Now we consider a horizontal ray with $r+1$ vertices $\nu_{i_{t}, j}, t = 0, 1, \dots, k$. Without
loss generalization, we assume $\nu_{i_{0}, j}$ is a boundary vertex. According to equation~\eqref{equ:edge}, we have
\begin{align}
\mu_{i_{1}, j}^{1}(x, y) - \mu_{i_{1}, j}^{2}(x, y) &= d_{i_{1}, j}(x, y)( x - x_{i_{1}})^{\alpha+1},, \nonumber\\
\dots \nonumber\\
\label{equ:ray}\mu_{i_{t}, j}^{1}(x, y) - \mu_{i_{t}, j}^{2}(x, y) &= d_{i_{t}, j}(x, y)( x - x_{i_{t}})^{\alpha+1}, \\
\dots \nonumber\\
0 - \mu_{i_{k}, j}^{2}(x, y) &= d_{i_{k}, j}(x, y)( x - x_{i_{k}})^{\alpha+1}.\nonumber
\end{align}
and $$\mu_{i_{t-1}, j}^{2}(x, y) = \mu_{i_{t}, j}^{1}(x, y).$$ Sum all these
equations, we have the following equation,
\begin{equation}
\sum_{t=1}^{k} d_{i_{t}, j}(x, y)( x - x_{i_{t}})^{\alpha+1} =
\mu_{i_{1}, j}^{2}(x, y).
\end{equation}
The main different of the constraint from \eqref{equ:con0} and \eqref{equ:con1} is that these is one
edge cofactor $\mu_{i_{1}, j}^{2}(x, y)$ in the constraints. For any $d_{i_{t}, j}(x, y)$ assigned for
the interior vertices, we can assign $\mu_{i_{1}, j}^{2}(x, y)$ as $\sum_{t=1}^{k} d_{i_{t}, j}(x, y)( x - x_{i_{t}})^{\alpha+1}$
to satisfy the constraint.

For a horizontal cross-cut in the T-mesh, since it has two boundary vertices, we can conclude that it
has $(d_{1}+1)(d_{2}-\beta)$ degrees of freedoms. Similarly, a vertical cross-cut has $(d_{2}+1)(d_{1}-\alpha)$
degrees of freedoms.

The linear systems as \eqref{equ:con0} or \eqref{equ:con1}, associated with each interior l-edges can be assembled into
a global system as $\m{\bf x}=0$. Here $\m$ is a $n_{r}\times n_{c}$ matrix, which is
called \textbf{conformality conditions matrix}. And ${\bf x}$ is a column vector whose
elements are all the vertex cofactors for the interior vertices in the T-mesh.

And according to the above anlaysis, we can get the dimension for spline space over any general T-mesh without holes according to
smoothing cofactor-conformality method which is stated in the following theorem.
\begin{theorem}\label{the:dim1}Given a T-mesh $\mathcal{T}$ which has no vanished l-edges and holes,
let matrix $\m$ be the conformality conditions matrix, then the
dimension of spline space over the T-mesh is,
\begin{align*}
\dim\mathcal{S}(d_{1}, d_{2}, \alpha, \beta, \mathcal{T}) = & (d_{1}
+ 1)(d_{2} + 1) + C^{h}(d_{1} + 1)(d_{2} - \beta) + \nonumber \\
& C^{v}(d_{2} + 1)(d_{1} - \alpha) + V^{+}(d_{1} - \alpha)(d_{2} - \beta) + \dim \m,
\end{align*}
\end{theorem}
\section{Diagonalizable}
Theorem~\ref{the:dim1} indicates that the main difficult to compute the
dimension of spline space over T-mesh is to calculate the rank of conformality
condition matrix $\m$. It is obvious that the structure of the matrix is associated with the order of edge
conformality conditions and it is also associated with the order of
vertices cofactors. Most of existing method study the dimension of the matrix by forcing the T-mesh to be nest structure,
such as T-subdivision, hierarchical. However, nesting structure should be fine for application, but
it is hiding some essential properties for T-mesh. What we wish to answer is that under what condition can we
compute the dimension regardless the knot intervals for a given T-mesh and why such condition is important.

\begin{definition}
Given a T-mesh $\mt$, suppose we order the all the interior l-edges
as $e_{i_{j}}, j = 1, 2, \dots, n_{e}$, then we can compute the \textbf{new-vertex-vector} $\nu^{i}$.
Here $\nu^{i}_{1}$ is the number of vertices on l-edge $e_{i_{1}}$ and $\nu^{i}_{j}$ is the number
of vertices on l-edge $e_{i_{j}}$ but not on l-edges $e_{i_{k}}, k < j$.
\end{definition}
\begin{definition}\label{definition1}A T-mesh is called
\textbf{Diagonalizable} if there exists an order of l-edges $e_{i_{j}}, j = 1, 2, \dots, n_{e}$
such that the \emph{new-vertex-vector} $\nu^{i}$ satisfies that $\nu_{j}^{i} \geq N^{h}$ if $e_{i_{j}}$
is horizonal and $\nu_{j}^{i} \geq N^{v}$ if $e_{i_{j}}$ is vertical.
\end{definition}

\begin{lemma}\label{lemma:four}If a T-mesh is diagonalizable, then the matrix
$\m$ has full column rank regardless the knot intervals.
\end{lemma}
\begin{proof}
Since we assume that the T-mesh has no vanished l-edges, the matrix $\m$ has more columns than
rows, i.e., $n_{c} \leq n_{r}$. After
arranging the order of edge conformality conditions and the order of
vertex cofactors, an appropriate partition of the linear system of
constraints is
 \begin{equation}
    \left[\begin{array}{c|c}
        \m_1 & \m_2
      \end{array}\right]
    \left[
      \begin{array}{c}
        \textbf{x}_{1} \vspace{2pt} \\
        \hline \vspace{-10pt} \\
        \textbf{x}_2
      \end{array}\right] =
    \mathbf{0}
  \end{equation}
where $\m_1$ is a $n_{r} \times n_{r}$ matrix and $\m_2$ is a
$n_{r} \times (n_{c} - n_{r})$ matrix,
$\textbf{x}_1^T$ is a vector of the first
$n_{r}$ vertex cofactors, and $\textbf{x}_2$ is
a vector of the remaining vertex cofactors.

Since the T-mesh is diagonalizable, so there exist the order for the interior
l-edges satisfy the condition stated in definition~\ref{definition1}. Without loss of generalization, we assume the order
for the l-edges are $e_{i}, i = 1,2, \ldots, n_{e}$. Then we
arrange the order of the edge conformality conditions corresponding
to l-edge from $e_{n_{e}}$ to $e_{1}$. The order of the vertex
cofactors can be placed in the following fashion. For l-edge $e_{n_{e}}$, if it is
a horizonal l-edge, then there exist $N^{h}$ vertices which are not appeared in
the other l-edges. Each vertex corresponds $(d_{1} - \alpha)\times(d_{2} - \beta)$ cofactors, so these
$N^{h}$ vertices correspond $(d_{1} - \alpha)\times(d_{2} - \beta)N^{h} \geq (d_{1} + 1)(d_{2}-\beta)$ cofactors. Select
any $(d_{1} + 1)(d_{2}-\beta)$ cofactors and put in the beginning of $\textbf{x}_1^T$. If the l-edge is a vertical
l-edge, we will select $(d_{2}+1)(d_{1} - \alpha)$ cofactors and put them in the beginning of $\textbf{x}_1^T$. These
process can be applied for the remaining l-edges. And then the matrix $\m_1$
is in upper block triangular form and according to
Lemma~\ref{lemma:dim} each diagonal block $(d_{2}-\beta)(d_{1} + 1)\times(d_{2}-\beta)(d_{1} + 1)$ or
$(d_{1}-\alpha)(d_{2} + 1)\times(d_{1}-\alpha)(d_{2} + 1)$ matrix is full rank, thus matrix $\m_1$ is obviously of full
rank.
\end{proof}

\begin{theorem}\label{the:dim2}Suppose T-mesh $\mathcal{T}$ is diagonalizable and
has no vanished l-edges and holes, then the
dimension of spline space over the T-mesh is,
\begin{align*}
\dim\mathcal{S}(d_{1}, d_{2}, \alpha, \beta, \mathcal{T}) = & (d_{1}
+ 1)(d_{2} + 1) + (C^{h} - T^{h})(d_{1} + 1)(d_{2} - \beta) + \nonumber \\
& (C^{v} - T^{v})(d_{2} + 1)(d_{1} - \alpha) + V(d_{1} - \alpha)(d_{2} - \beta).
\end{align*}
\end{theorem}
\begin{proof}As T-mesh $\mt$ is diagonalizable, so $$\dim \m = n_{c} - n_{r} = (d_{1} - \alpha)(d_{2} - \beta)(V - V^{+}) -
(d_{1}+1)(d_{2}-\beta)E^{h} - (d_{2}+1)(d_{1}-\alpha)E^{v}.$$

According to theorem~\ref{the:dim1}, we have the dimension $\dim\mathcal{S}(d_{1}, d_{2}, \alpha, \beta, \mathcal{T})$ is
\begin{align*}
dim=& (d_{1} + 1)(d_{2} + 1) + C^{h}(d_{1} + 1)(d_{2} - \beta) + C^{v}(d_{2} + 1)(d_{1} - \alpha) + V^{+}(d_{1} - \alpha)(d_{2} - \beta) + \nonumber\\
&(d_{1} - \alpha)(d_{2} - \beta)(V - V^{+}) - (d_{1}+1)(d_{2}-\beta)E^{h} - (d_{2}+1)(d_{1}-\alpha)E^{v} \nonumber \\
=&(d_{1} + 1)(d_{2} + 1) + (C^{h} - T^{h})(d_{1} + 1)(d_{2} - \beta) + (C^{v} - T^{v})(d_{2} + 1)(d_{1} - \alpha) + \nonumber \\
&V(d_{1} - \alpha)(d_{2} - \beta).
\end{align*}
\end{proof}

Now we provide several examples for dimension of spline space over the following T-meshes.
\begin{figure}[htbp]
\centering
~\\[-1ex]
\subfigure [a.]{\includegraphics [width=2.3in]{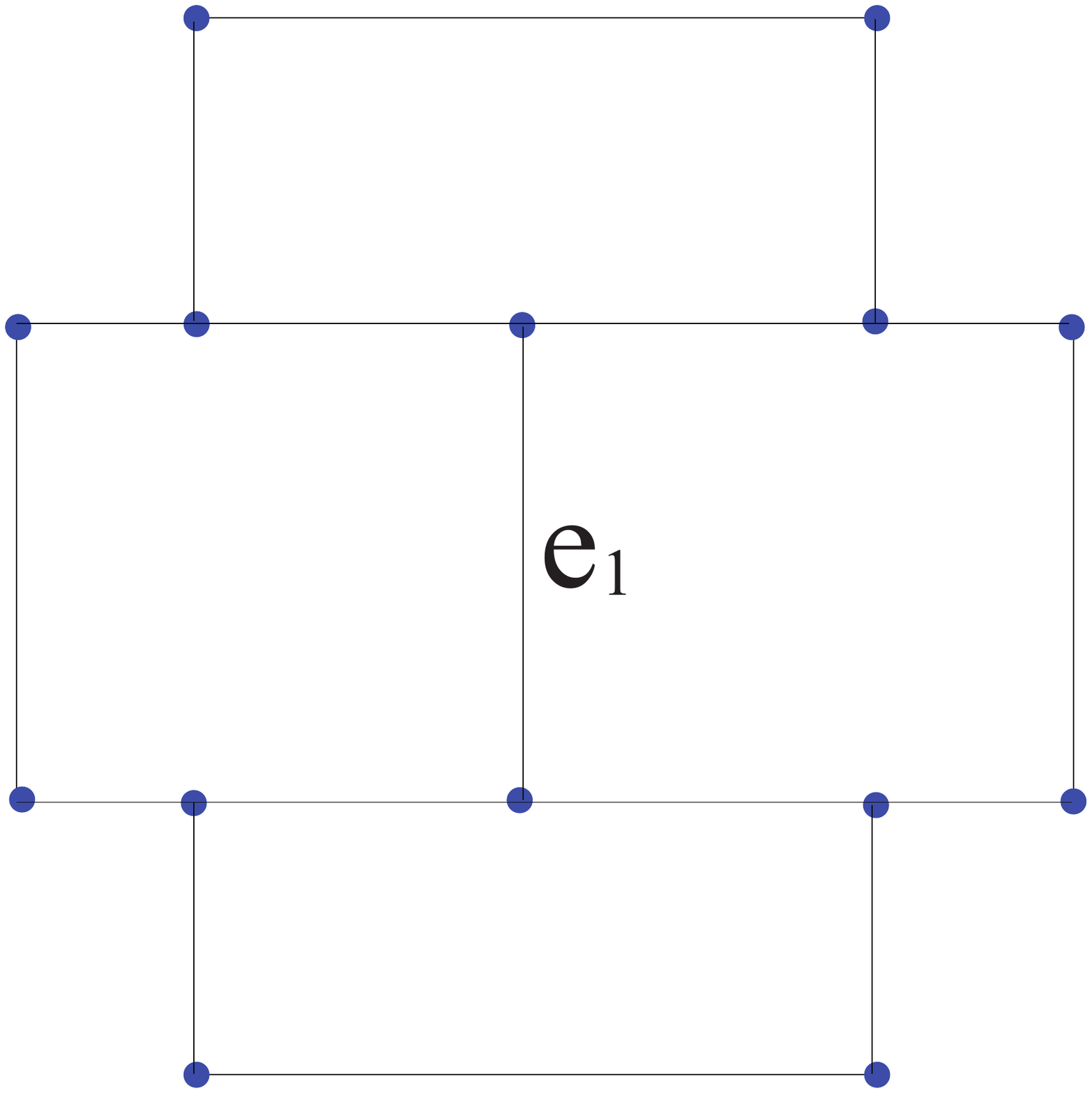}}
\subfigure [b.]{\includegraphics [width=2.3in]{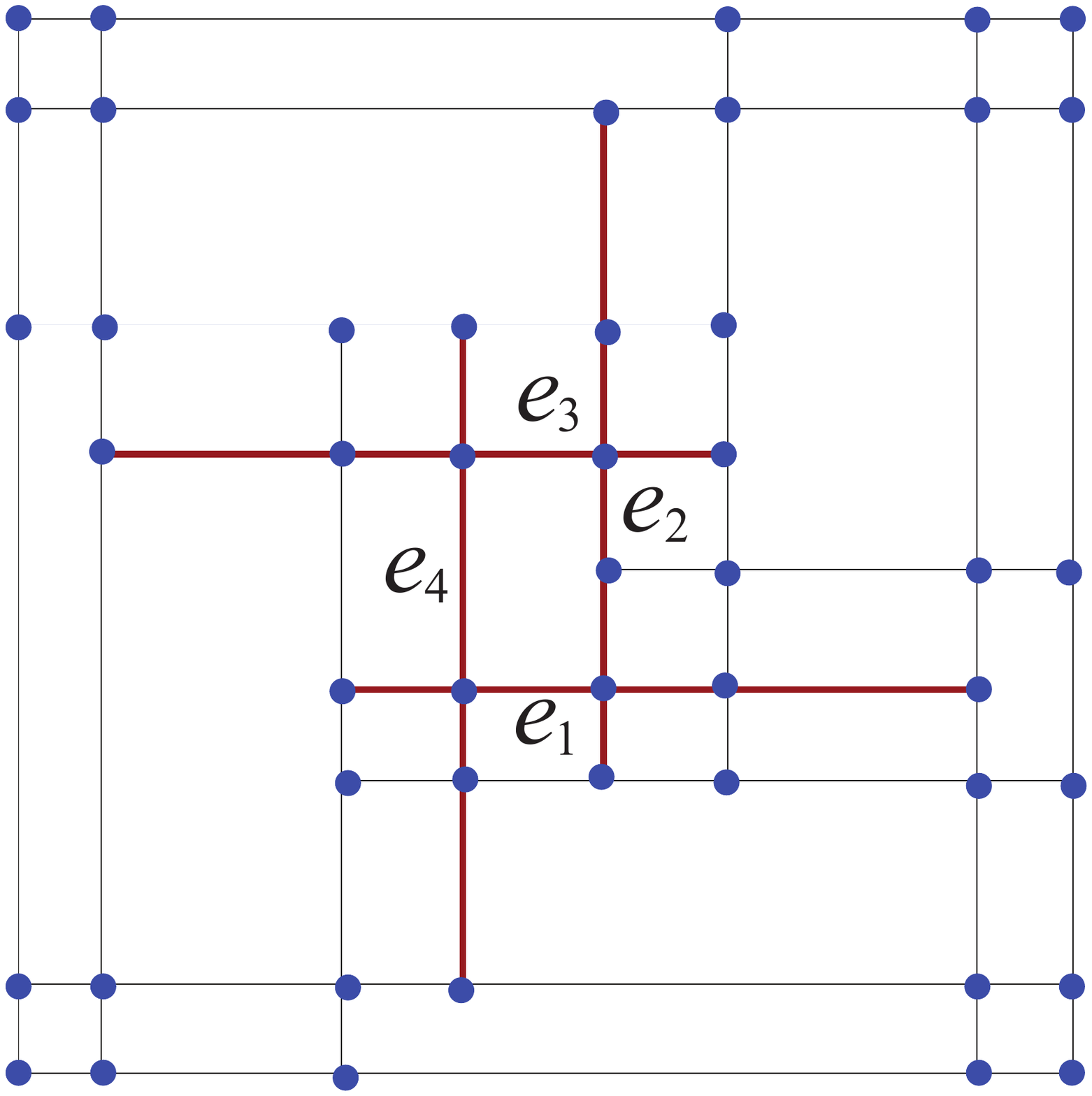}}
~\\[-1ex]
\caption{Two example T-meshes for spline space $\mathcal{S}(3,3,2,2,\mathcal{T})$.} \label{fig:dia_ex1}
\end{figure}
\begin{example}The first examples are spline space $\mathcal{S}(3,3,2,2,\mathcal{T})$ over the two
T-meshes in Figure~\ref{fig:dia_ex1}.

The first T-mesh has some concave corners and one interior l-edge
which is a vanished l-edge since it has only two vertices. And the T-mesh has two cross-cut. So according to
theorem~\ref{the:dim1}, the dimension of the spline space $\mathcal{S}(3,3,2,2,\mathcal{T})$ over the first T-mesh
is $16 + 2\times4 = 24$.

The second T-mesh is much more complex. It has four interior l-edges $e_{i}, i = 1, \dots, 4$, four cross-cuts and five rays.
If we arrange l-edges as $e_{1}$, $e_{2}$, $e_{3}$, $e_{4}$, then the new-vertex-vector corresponds this order is
$(5,5,4,3)$. But we can arrange the l-edges as order $e_{1}$, $e_{4}$, $e_{3}$, $e_{2}$, then the new-vertex-vector
corresponds this order is $(5,4,4,4)$, which means the T-mesh is diagonalizable. In the T-mesh, we have $4$ cross-cuts,
$2$ horizonal and $2$ vertical interior l-edges and $31$ interior vertices, so according to
theorem~\ref{the:dim2}, the dimension of the spline space $\mathcal{S}(3,3,2,2,\mathcal{T})$ over the first T-mesh
is $16 + 4(4-4) + 31= 47$. We should mention that non of existing method can calculate the dimension for this example since the
T-mesh is not T-subdivision or hierarchical.
\end{example}

\begin{figure}[htbp]
\centering
~\\[-1ex]
\subfigure {\includegraphics [width=2.3in]{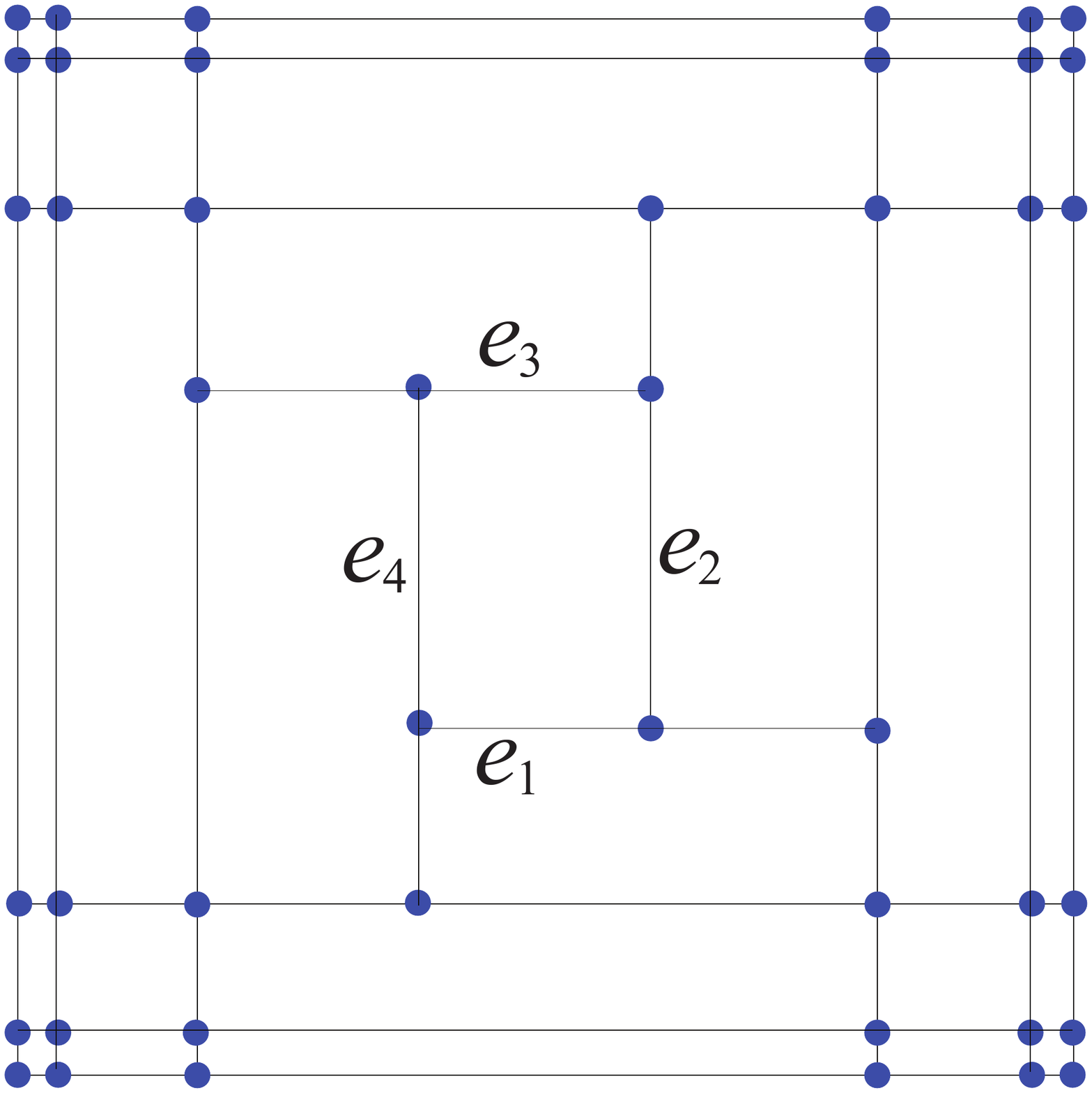}}
\subfigure {\includegraphics [width=2.3in]{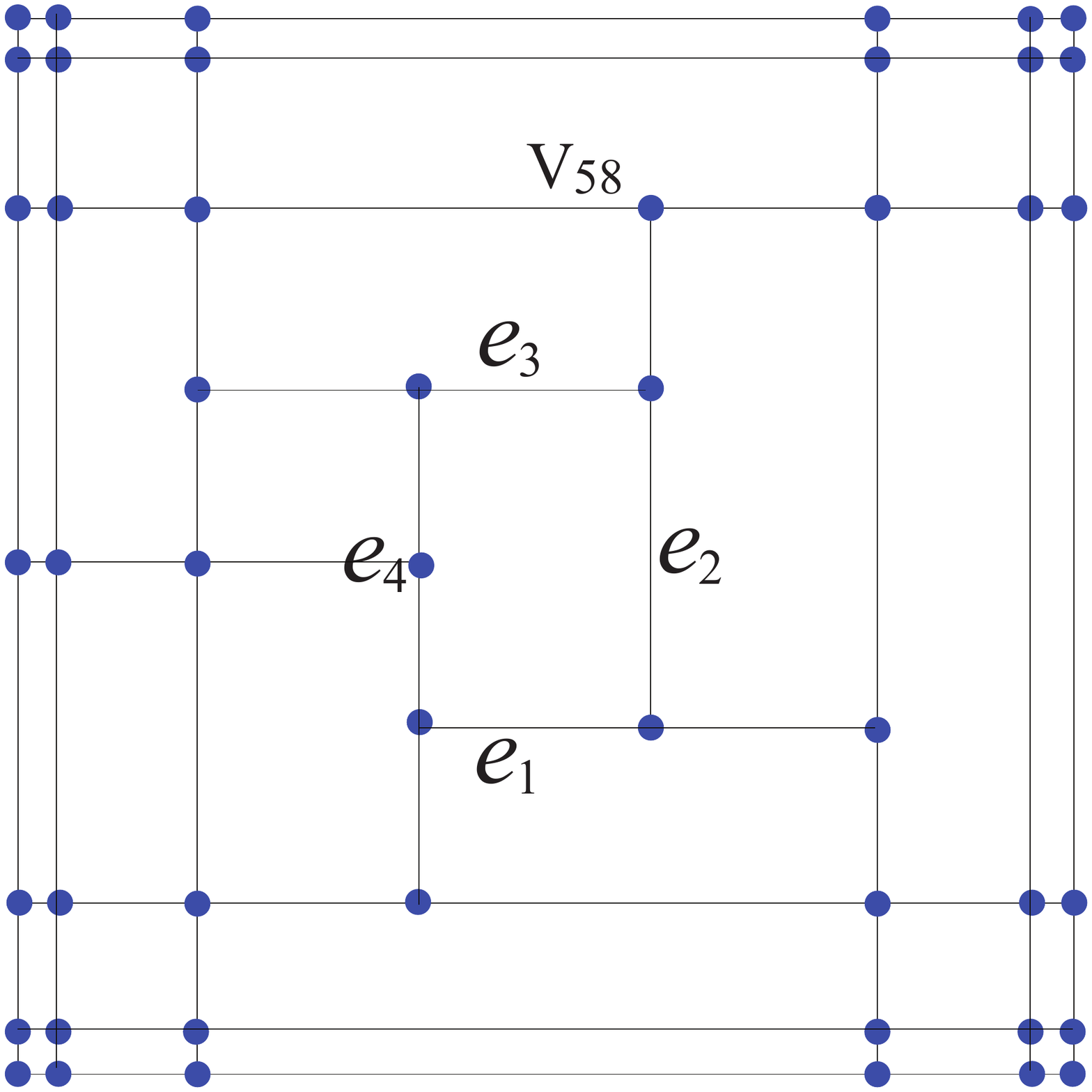}}
~\\[-1ex]
\caption{Non-diagonalizable vs. diagonalizable for spline space $\mathcal{S}(3,3,1,1,\mathcal{T})$.} \label{fig:dia_ex}
\end{figure}
\begin{example}The second example is associated with spline space $\mathcal{S}(3,3,1,1,\mathcal{T})$ over the two
T-meshes in Figure~\ref{fig:dia_ex}.

The first T-mesh has four interior l-edges. And we can see that it is not diagonalizable since the new-vertex-vector
could be $(3,2,2,1)$ or $(3,3,1,1)$ for different orders. Thus, we cannot complete the dimension
using the current method. Actually, according to our knowledge, no existing method can compute
the dimension of such spline space.

The second T-mesh also has four interior l-edges, but according to our knowledge,
no existing method can compute the dimension of such spline space. But if we arrange
the order of the l-edges to be $e_1$, $e_2$, $e_3$ and $e_4$, we can see that the new-vertex-vector
is $(3,2,2,2)$, i.e., the T-mesh is diagonalizable. Since the T-mesh has 8 cross-cut, $4$ interior l-edges, and
$27$ interior vertices. So the dimension of the spline space $\mathcal{S}(3,3,1,1,\mathcal{T})$
over the second T-mesh is $16 + 4\times2\times(8-4) + 27\times4= 156$.
\end{example}

\begin{figure}[htbp]
\centering
~\\[-1ex]
\subfigure {\includegraphics [width=2.3in]{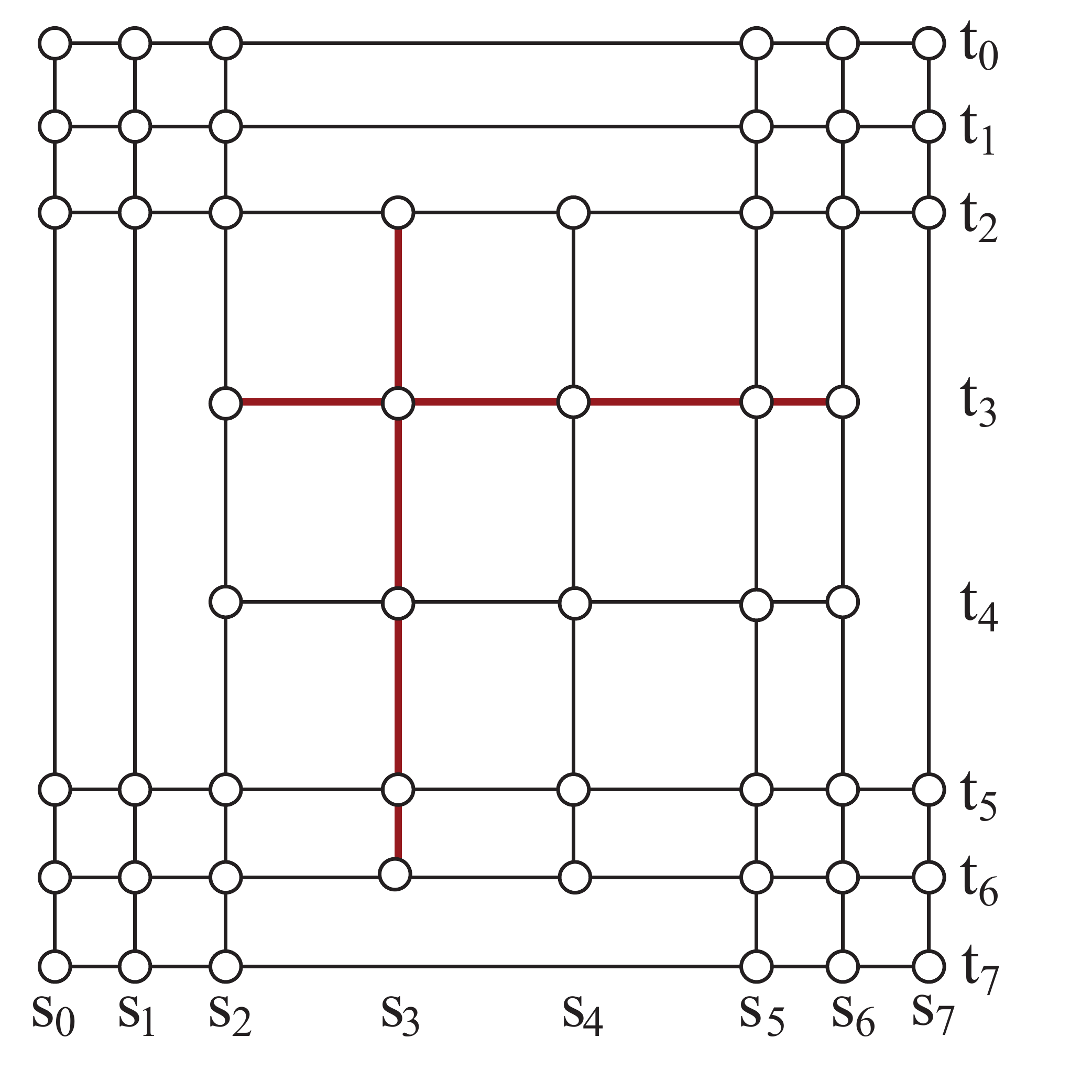}}
\subfigure {\includegraphics [width=2.3in]{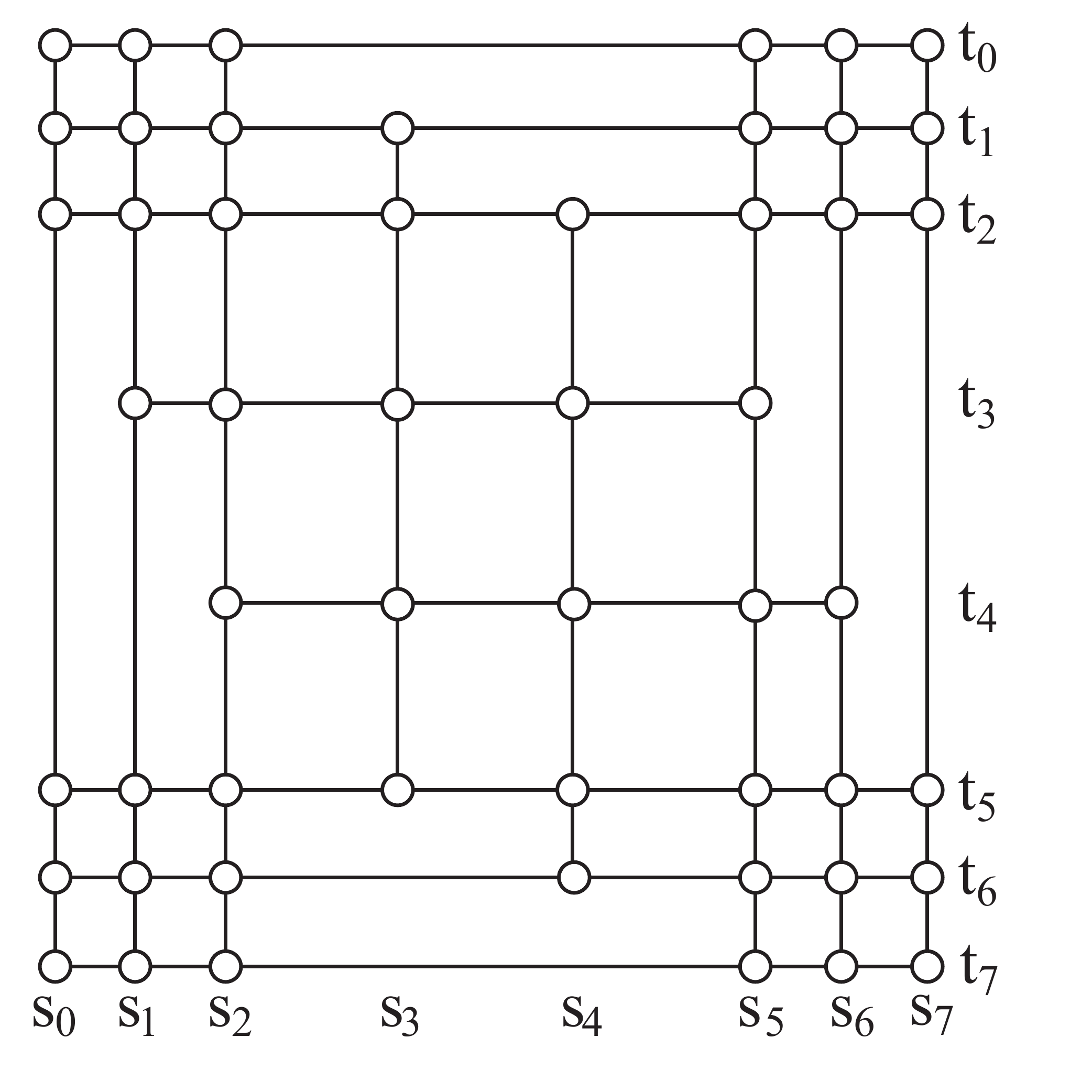}}
~\\[-1ex]
\caption{The dimension of bi-cubic spline space over the T-mesh
is instable, which is generated from the left T-mesh by moving two red l-edges.} \label{fig:dia_ex2}
\end{figure}

\begin{example}The third example is associated with spline space $\mathcal{S}(3,3,2,2,\mathcal{T})$ over the two
T-meshes in Figure~\ref{fig:dia_ex2}. Both T-meshes have four interior l-edges and we can check that the T-meshes
are not diagonalizble. For the first T-mesh, the dimension is $65$ and the dimension for the second T-mesh is
instable, i.e., the dimension is associated with the value of the knots.

We can see that the second T-mesh can be constructed by moving two red l-edges of the first T-mesh. Although the
two T-meshes have different structure, but if we look at the structure of the conformality conditions matrix, we
can see that the structure of two matrixes are identical, except the two red l-edges using the knots of $s_{2}$
to $s_{6}$ instead of $s_{1}$ to $s_{5}$. So diagonalizble is a concept for the structure of the
conformality conditions matrix not T-mesh itself.

This example also tells us that if a T-mesh is not diagonalizable, then we have to consider the knot values in order
to analysis the dimension of the spline space using smoothing cofactor-conformality method.
\end{example}

\begin{remark}A similar result has been abstained in~\cite{Mourrain}, which provides
the dimension for a special T-mesh, called regular T-subdivision. The main difference between
these two method is that the diagonalizable T-meshes don't need to be nested structure.
Regular T-subdivision is a special case of diagonalizable T-mesh.
\end{remark}

\subsection{Characterization}
In this section, we will provide a necessary and sufficient condition for characterization
a diagonalizable T-mesh.
\begin{lemma}\label{lemma:dia}A a necessary and sufficient condition for a T-mesh to be
diagonalizable is for any interior l-edges set $\mathbb{S}$, there at least exists one
horizonal l-edge such that the number of vertices on this l-edge but not on
the other l-edge in $\mathbb{S}$ is at least $N^{h}$, or
there at least exists one vertical l-edge such that the number of vertices on this l-edge but not on
the other l-edge in $\mathbb{S}$ is at least $N^{v}$.
\end{lemma}
\begin{proof}First, we prove the condition is necessary using reduction to absurdity. IF the T-mesh is
diagonalizable, but there exists a set of l-edges $\{ e_{i_{1}}, e_{i_{2}}, \dots, d_{i_{s}}\}$ such that any
horizonal l-edges in the set at most have $N^{h}$ vertices which are not on the other l-edges in the set and any
vertical l-edges at most have $N^{v}$ vertices which are not on the other l-edges in the set. And since the T-mesh
is diagonalizable, so without loss of generalization, we assume when the order of the interior l-edges is
$e_{1}, e_{2}, \dots, e_{n_{e}}$, in which order the l-edges satisfy the condition of diagonalizable. Let $k$ to the
maximal index for all $i_{j}, j = 1, \dots, s$. Now, we consider l-edge $e_{k}$, since it has at most $N^{h}-1$ or $N^{v}-1$
vertices which are not on the other l-edges in the set, so it also has at most $N^{h}-1$ or $N^{v}-1$
vertices are not on the l-edges for $e_{1}, \dots, e_{k-1}$ since $\mathbb{S} \subseteq \{e_{1}, \dots, e_{k-1}, e_{k}\}$,
which violates the assumption of diagonalizable. Thus, the condition is necessary.

Now we prove the condition is sufficient. For the set of l-edges $e_{i}, i = 1, \dots, n_{e}$,
according to the assumption, there exist one l-edge which has enough vertices on the l-edge but not on
the others. Without loss of generalization, we assume it is $e_{1}$. Suppose we have ordered the l-edges
as $e_{1}, e_{2}, \dots, e_{j}$ satisfy the diagonalizable condition, then for set $\{e_{j+1}, \dots, e_{n_{e}}\}$, according to
the assumption,  there exist one l-edge which has enough vertices on the l-edge but not on
the others. Without loss of generalization, we assume it is $e_{j+1}$. With this process, we can order the l-edges such that it
satisfy the diagonalizable condition, which completes the proof.
\end{proof}

\section{Correction for~\cite{LiWaZh06}}
In this section, we will show that the dimension result in~\cite{LiWaZh06}
is not right. Precisely, the dimension under the condition of~\cite{LiWaZh06}
is possible instable. And we also provide a new modified theorem using
the technology builded in the last section.

\subsection{A instable example under the condition of~\cite{LiWaZh06}}
\begin{figure}[htbp]
\centering
~\\[-1ex]
\subfigure {\includegraphics [width=3.0in]{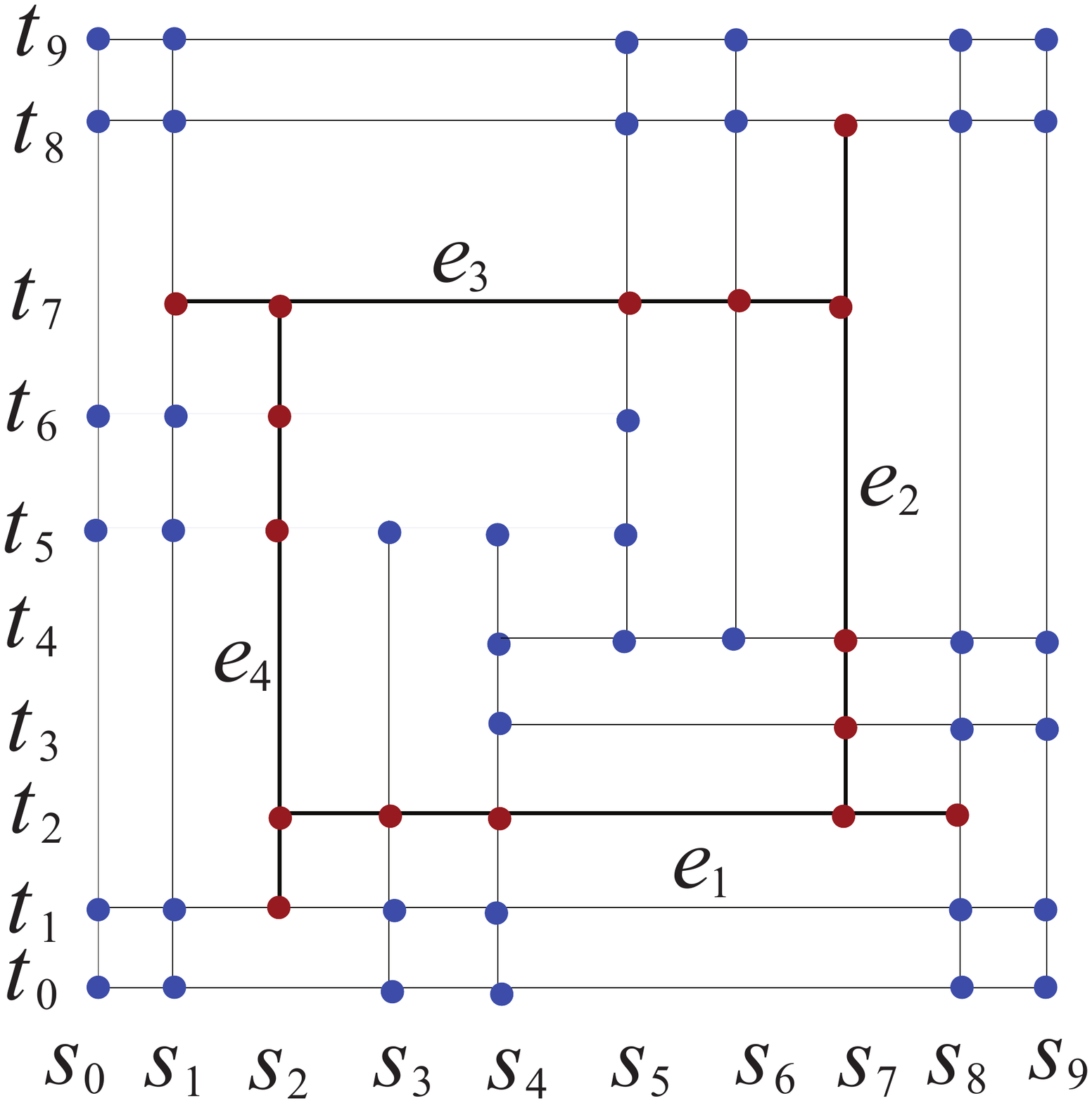}}
~\\[-1ex]
\caption{Counterexample for paper~\cite{LiWaZh06}.} \label{fig:counter}
~\\[-1ex]
\end{figure}

Consider the spline space $\mathcal{S}(3,3,2,2,\mathcal{T})$ over the T-mesh illustrated in
Figure~\ref{fig:counter}. There are four interior l-edges in the T-mesh and each one have two
mocro-vertices in the interior, which satisfy the condition in~\cite{LiWaZh06}. However, we will analysis
the dimension in the following and show that the dimension of the spline space over the T-mesh
is instable.

Actually, we arrange the interior l-edges as the order of $e_{1}$, $e_{2}$,
$e_{3}$ and $e_{4}$. And we arrange the order of the vertices as
$v_{2,2}$, $v_{3,2}$, $v_{4,2}$, $v_{8,2}$, $v_{7,2}$, $v_{7,3}$, $v_{7,4}$, $v_{7,8}$, $v_{7,7}$, $v_{6,7}$,
$v_{5,7}$, $v_{1,7}$, $ v_{2,7}$, $v_{2,6}$, $v_{2,5}$, $v_{2,1}$ and $v_{2,2}$, then we can
get the sparse matrix $\m$ which has the following form,
$$\m = \left(
        \begin{array}{cccccccccccccccc}
          1 & 1 & 1 & 1 & 1 & 0 & 0 & 0 & 0 & 0 & 0 & 0 & 0 & 0 & 0 & 0 \\
          s_{2} & s_{3} & s_{4} & s_{8} & s_{7} & 0 & 0 & 0 & 0 & 0 & 0 & 0 & 0 & 0 & 0 & 0 \\
          s_{2}^2 & s_{3}^2 & s_{4}^2 & s_{8}^2 & s_{7}^2 & 0 & 0 & 0 & 0 & 0 & 0 & 0 & 0 & 0 & 0 & 0 \\
          s_{2}^3 & s_{3}^3 & s_{4}^3 & s_{8}^3 & s_{7}^3 & 0 & 0 & 0 & 0 & 0 & 0 & 0 & 0 & 0 & 0 & 0 \\
          0 & 0 & 0 & 0 & 1 & 1 & 1 & 1 & 1 & 0 & 0 & 0 & 0 & 0 & 0 & 0 \\
          0 & 0 & 0 & 0 & t_{2} & t_{3} & t_{4} & t_{8} & t_{7} & 0 & 0 & 0 & 0 & 0 & 0 & 0 \\
          0 & 0 & 0 & 0 & t_{2} & t_{3}^2 & t_{4}^2 & t_{8}^2 & t_{7}^2 & 0 & 0 & 0 & 0 & 0 & 0 & 0 \\
          0 & 0 & 0 & 0 & t_{2} & t_{3}^3 & t_{4}^3 & t_{8}^3 & t_{7}^3 & 0 & 0 & 0 & 0 & 0 & 0 & 0 \\
          0 & 0 & 0 & 0 & 0 & 0 & 0 & 0 & 1 & 1 & 1 & 1 & 1 & 0 & 0 & 0 \\
          0 & 0 & 0 & 0 & 0 & 0 & 0 & 0 & s_{7} & s_{6} & s_{5} & s_{1} & s_{2} & 0 & 0 & 0 \\
          0 & 0 & 0 & 0 & 0 & 0 & 0 & 0 & s_{7}^2 & s_{6}^2 & s_{5}^2 & s_{1}^2 & s_{2}^2 & 0 & 0 & 0 \\
          0 & 0 & 0 & 0 & 0 & 0 & 0 & 0 & s_{7}^3 & s_{6}^3 & s_{5}^3 & s_{1}^3 & s_{2}^3 & 0 & 0 & 0 \\
          1 & 0 & 0 & 0 & 0 & 0 & 0 & 0 & 0 & 0 & 0 & 0 & 1 & 1 & 1 & 1 \\
          t_{2} & 0 & 0 & 0 & 0 & 0 & 0 & 0 & 0 & 0 & 0 & 0 & t_{7} & t_{6} & t_{5} & t_{1} \\
          t_{2}^2 & 0 & 0 & 0 & 0 & 0 & 0 & 0 & 0 & 0 & 0 & 0 & t_{7}^2 & t_{6}^2 & t_{5}^2 & t_{1}^2 \\
          t_{2}^3 & 0 & 0 & 0 & 0 & 0 & 0 & 0 & 0 & 0 & 0 & 0 & t_{7}^3 & t_{6}^3 & t_{5}^3 & t_{1}^3 \\
        \end{array}
      \right)
$$

If $s_{i} = i$ and $t_{j} = j$, we can verify that the rank of the matrix is $15$, thus the dimension of
the spline space is $49$. If we perturb one of the knots a little bit, such as $s_{3} = 3.0 + \epsilon$,
where $\epsilon$ is an arbitrary small value, then the dimension is $48$. In other words, the dimension is instable.

\subsection{Correction theorem}
In this section, we will provide a new theorem to correct the theorem in~\cite{LiWaZh06} using the method in the last section.
\begin{theorem}Given a regular T-mesh $\mt$ without holes, if each
horizontal l-edges has at least $N^{h} - 1$ mono-vertices and each vertical l-edge has at least
$N^{v} - 1$ mono-vertices except two end vertices, then the dimension of the spline
space defined on $\mt$ is
\begin{align*}
\dim\mathcal{S}(d_{1}, d_{2}, \alpha, \beta, \mathcal{T}) = & (d_{1}
+ 1)(d_{2} + 1) + (C^{h} - T^{h})(d_{1} + 1)(d_{2} - \beta) + \nonumber \\
& (C^{v} - T^{v})(d_{2} + 1)(d_{1} - \alpha) + V(d_{1} - \alpha)(d_{2} - \beta).
\end{align*}
\end{theorem}

\begin{figure}[htbp]
\centering
~\\[-1ex]
\subfigure {\includegraphics [width=2in]{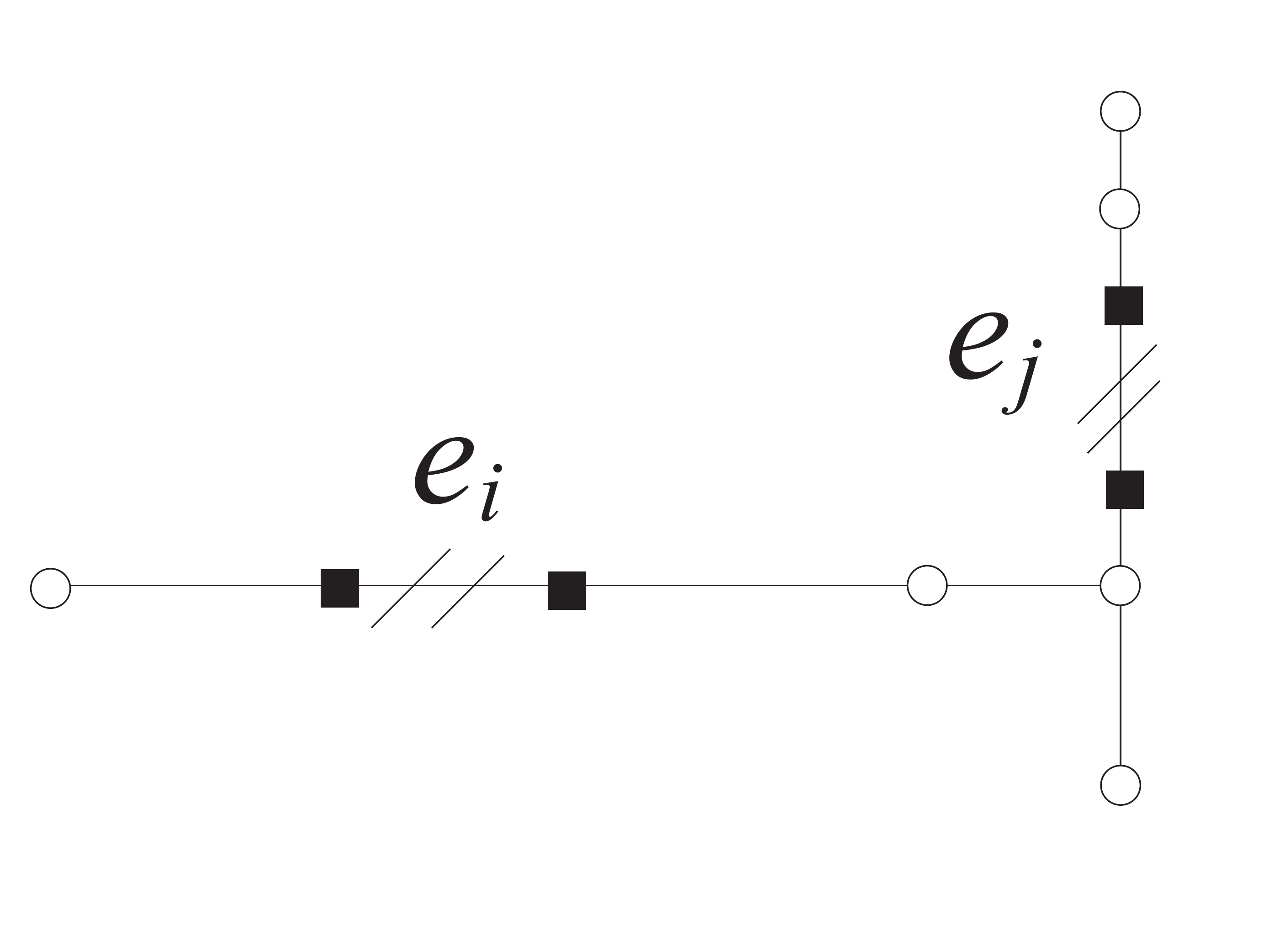}}
~\\[-1ex]
\caption{The new correction theorem for~\cite{LiWaZh06}.} \label{fig:proof}
~\\[-1ex]
\end{figure}
\begin{proof}We will prove that under the condition, the T-mesh is diagonalizable using reduction to absurdity.
Actually, if the T-mesh is not diagonalizable. Then for any set of l-edges, it is bounded. Suppose $e_{i}$ is
the most bottom horizonal l-edges in the set (if there are
more than one l-edge in the set, we will pick the leftmost one), see Figure~\ref{fig:proof} as an illustration.
According to the assumption, $e_{i}$ has at
least $N^{h} - 1$ mono-vertices (black rectangle vertices in the figure), if one of the two end vertices is not on
the other l-edges of in the set, then $e_{i}$ at least has $N^{h}$ vertices which are not on the other l-edges
in the set. According to Lemma~\ref{lemma:dia}, the T-mesh is diagonalizable which violates the assumption. Thus,
there exists a vertical l-edge, $e_{j}$, which contains
one of the end vertices of $e_{i}$. Without loss of generalization, we assume the right end vertex is on the l-edge
$e_{j}$ in the set. Now, we consider l-edge $e_{j}$, the bottom end vertex of the l-edge cannot lie on the other
l-edges in the set because it is on the bottom of $e_{i}$ which is the bottommost l-edges in the set. So $e_{j}$ at least has
$N^{v}$ vertices which are not on the other l-edges in the set. According to Lemma~\ref{lemma:dim}, the T-mesh is diagonalizable which violates the assumption.
Using Theorem~\ref{the:dim2}, we prove the theorem.
\end{proof}

\section{Conclusion and Future work}
In the present paper, we introduce a class of T-meshes, diagonalizable T-meshes, over which
the dimension of spline spaces is stable. We also provide a necessary and sufficient condition
to characterize this class of T-meshes. The dimension result in the present paper can cover all
the existing dimension results as special cases.

The paper leaves several open problems for further research. As we
have provided the dimension of the spline space, so there are many
problems which need to be solved, such as construction of a set of
basis functions with good properties, geometric operations and
properties of the splines over T-meshes, etc. We will explore these
problems in detail in future papers. It is also
an important and interesting question to find out other general with
fix dimension T-meshes.

\section*{Acknowledgements}
The authors are supported by the NSF of China
(No.11031007, No.60903148), Chinese Universities Scientific Fund,
SRF for ROCS SE and Chinese Academy of Science (Startup Scientific
Research Foundation).

\bibliographystyle{elsarticle-num}
\bibliography{lixin_bi}

\begin{thebibliography}{10}
\expandafter\ifx\csname url\endcsname\relax
  \def\url#1{\texttt{#1}}\fi
\expandafter\ifx\csname urlprefix\endcsname\relax\def\urlprefix{URL }\fi
\expandafter\ifx\csname href\endcsname\relax
  \def\href#1#2{#2} \def\path#1{#1}\fi

\bibitem{LiWaZh06}
C.~J. Li, R.~H. Wang, F.~Zhang, {Improvement on the Dimensions of Spline Spaces
  on T-Mesh}, Journal of Information \& Computational Science 3~(2) (2006)
  235--244.

\bibitem{Farin02}
G.~Farin, NURBS {C}urves and {S}urfaces: from {P}rojective {G}eometry to
  {P}ractical {U}se, Fourht {E}dition, A. K. Peters, Ltd., Natick, MA, 2002.

\bibitem{Cottrell:2009rp}
J.~A. Cottrell, T.~J.~R. Hughes, Y.~Bazilevs, Isogeometric analysis: Toward
  {I}ntegration of {CAD} and {FEA}, Wiley, Chichester, 2009.

\bibitem{htspline1}
J.~Deng, F.~Chen, Y.~Feng, Dimensions of spline spaces over t-meshes, Journal
  of Computational and Applied Mathematics 194 (2006) 267--283.

\bibitem{htspline2}
J.~Deng, F.~Chen, X.~Li, C.~Hu, W.~Tong, Z.~Yang, Y.~Feng, Polynomial splines
  over hierarchical t-meshes, Graphical Models 74 (2008) 76--86.

\bibitem{htspline3}
X.~Li, J.~Deng, F.~Chen, Surface modeling with polynomial splines over
  hierarchical t-meshes, The Visual Computer 23 (2007) 1027--1033.

\bibitem{htspline4}
X.~Li, J.~Deng, F.~Chen, Polynomial splines over general t-meshes, The Visual
  Computer 26 (2010) 277--286.

\bibitem{iga_pht1}
N.~Nguyen-Thanh, H.~Nguyen-Xuan, S.~P.~A. Bordas, T.~Rabczuk, Isogeometric
  analysis using polynomial splines over hierarchical t-meshes for
  two-dimensional elastic solids, Computer Methods in Applied Mechanics and
  Engineering 200 (2011) 1892¨C1908.

\bibitem{iga_pht2}
P.~Wang, J.~Xu, J.~Deng, F.~Chen, Adaptiveisogeometricanalysis using
  rationalpht-splines, Computer-Aided Design 43 (2011) 1438--1448.

\bibitem{pht_pde}
L.~Tian, F.~Chen, Q.~Du, Adaptive finite element methods for elliptic equations
  over hierarchical t-meshes, J. Comput. Appl. Math. 236 (2011) 878--891.

\bibitem{Schumaker12_ca}
L.~L. Schumaker, L.~Wang, Splines on triangulations with hanging vertices,
  Constructive Approximation\href {http://dx.doi.org/10.1007/s00365-012-9167-x}
  {\path{doi:10.1007/s00365-012-9167-x}}.

\bibitem{SeZhBaNa03}
T.~W. Sederberg, J.~Zheng, A.~Bakenov, A.~Nasri, T-splines and {T-NURCCS}s, ACM
  Transactions on Graphics 22 (3) (2003) 477--484.

\bibitem{SeCaFiNoZhLy04}
T.~W. Sederberg, D.~L. Cardon, G.~T. Finnigan, N.~S. North, J.~Zheng, T.~Lyche,
  {T}-spline simplification and local refinement, ACM Transactions on Graphics
  23 (3) (2004) 276--283.

\bibitem{Sederberg08}
T.~W. Sederberg, G.~T. Finnigan, X.~Li, H.~Lin, H.~Ipson, Watertight trimmed
  {NURBS}, ACM Transactions on Graphics 27~(3) (2008) Article no. 79.

\bibitem{Bazilevs2009}
Y.~Bazilevs, V.~M. Calo, J.~A. Cottrell, J.~A. Evans, T.~J.~R. Hughes,
  S.~Lipton, M.~A. Scott, T.~W. Sederberg, Isogeometric analysis using
  {T}-splines, Computer Methods in Applied Mechanics and Engineering 199~(5-8)
  (2010) 229 -- 263.

\bibitem{BaBeCoHuSa06}
Y.~Bazilevs, L.~{Beirao de Veiga}, J.~Cottrell, T.~Hughes, G.~Sangalli,
  Isogeometric analysis: approximation, stability and error estimates for
  $h$-refined meshes, Mathematical Models and Methods in Applied Sciences 16
  (2006) 1031--1090.

\bibitem{LiZhSeHuSc10}
X.~Li, J.~Zheng, T.~W. Sederberg, T.~J.~R. Hughes, M.~A. Scott, On the linear
  independence of {T}-splines blending functions, Computer Aided Geometric
  Design, 29 (2012) 63--76.

\bibitem{ScLiSeHu10}
M.~A. Scott, X.~Li, T.~W. Sederberg, T.~J.~R. Hughes, Local refinement of
  analysis-suitable {T}-splines, Computer Methods in Applied Mechanics and
  Engineering 213-216 (2012) 206--222.

\bibitem{LiSc10}
X.~Li, M.~A. Scott, Analysis-suitable t-splines: Characterization,
  refinablility and approximation, submitted Mathematical Models and Methods in
  Applied Sciences for publish.

\bibitem{Schumaker12_nn}
L.~L. Schumaker, L.~Wang, Spline spaces on tr-meshes with hanging vertices,
  Numerische Mathematik 118 (2011) 531--548.

\bibitem{Mourrain}
B.~Mourrain, On the dimension of spline spaces on planar t-subdivisions, Arxiv
  preprint arXiv:1011.1752.

\bibitem{wangrh}
R.-H. Wang, Multivariate Spline Functions and Their Applications, Science
  Press/ Kluwer Academic Publishers, 2001.

\bibitem{sch_scm}
L.~Schumaker, On the dimension of spaces of piecewise polynomials in two
  variables, in: Multivariate Approximation Theory, In: Schempp, W., Zeller, K.
  (Eds.), Birkhauser Verlag, Basel, 1979, pp. 396--412.

\bibitem{Schumaker12_cagd}
L.~L. Schumaker, L.~Wang, Approximation power of polynomial splines on
  t-meshes, Computer Aided Geometric Design\href
  {http://dx.doi.org/http://dx.doi.org/10.1016/j.cagd.2012.04.003}
  {\path{doi:http://dx.doi.org/10.1016/j.cagd.2012.04.003}}.

\bibitem{buffa}
A.~Buffa, D.~Cho, M.~Kumar, Characterization of t-splines with reduced
  continuity order on t-meshes, Comput. Methods Appl. Mech. Engrg. 201-204
  (2012) 112--126.

\bibitem{xinli}
X.~Li, F.~Chen, On the instability in the dimension of spline space over
  particular t-meshes, Computer Aided Geometric Design 28 (2011) 420--426.

\bibitem{instablity2}
D.~Berdinskya, M.~jae Oha, T.~wan Kima, B.~Mourrain, On the problem of
  instability in the dimension of a spline space over a t-mesh, Computers
  Graphics 36(2) (2012) 507--513.

\bibitem{mengwu}
M.~Wu, J.~Deng, F.~Chen, The dimension of spline spaces with highest order
  smoothness over hierarchical t-meshes, Arxiv preprint arXiv:1112.1144.

\end{thebibliography}
\end{document}